\documentclass[11pt]{amsart}

\usepackage{amsfonts}
\usepackage{amssymb, latexsym, amsthm, amsmath, verbatim}
\usepackage{indentfirst}

\numberwithin{equation}{section}
\theoremstyle{definition}
\newtheorem{Thm}[equation]{Theorem}
\newtheorem{Prop}[equation]{Proposition}
\newtheorem{Cor}[equation]{Corollary}
\newtheorem{Lem}[equation]{Lemma}

\makeatletter
\def\imod#1{\allowbreak\mkern5mu{\operator@font mod}\,\,#1}
\makeatother

\begin{document}

\title[An Isomorphism]{An Isomorphism between Scalar-Valued Modular Forms and Modular Forms for Weil Representations}
\author[Yichao Zhang]{Yichao Zhang}
\address{Department of Mathematics, University of Connecticut, Storrs, CT 06269}
\email{yichao.zhang@uconn.edu}
\date{}
\subjclass[2010]{Primary: 11F41, 11F27}
\keywords{vector-valued, scalar-valued, modular form, Weil representation, weakly holomorphic, obstruction.}

\begin{abstract}
In this note, we consider discriminant forms that are given by the norm form of real quadratic fields and their induced Weil representations.
We prove that there exists an isomorphism between the space of vector-valued modular forms for the Weil representations that are invariant under the action of the automorphism group  and the space of scalar-valued modular forms that satisfy some $\epsilon$-condition, with which we translate Borcherds's theorem of obstructions to scalar-valued modular forms. In the end, we consider an example in the case of level $12$.
\end{abstract}

\maketitle

\section*{Introduction}

\noindent
Modular forms, or scalar-valued modular forms, have been studied extensively for over a century and become one of the central objects in number theory and other related fields. The spaces of scalar-valued modular forms behave well under multiplication, the action of Galois groups, and that of Hecke operators. On the other hand, vector-valued modular forms, which shall mean vector-valued modular forms associated to Weil representations throughout this note, were considered by Jacobi in their connection with theta functions of positive definite even lattices of even rank. They later naturally appeared when Borcherds \cite{borcherds1998automorphic} developed his theory of automorphic products, also known as the singular theta correspondence.  People have been trying to obtain satisfactory structure theory for spaces of vector-valued modular forms. For example, McGraw \cite{mcgraw2003rationality} considered the rationality and Bruinier and Stein \cite{bruinier2010weil} constructed Hecke operators for vector-valued modular forms. Nevertheless, the spaces of vector-valued modular forms do not possess as nice structures as that of scalar-valued modular forms.

It will be desirable then if one may go freely between these two types of modular forms, that is, if there is a one-to-one correspondence between them. In both directions, we have nice candidates: from the space of vector-valued modular forms to that of scalar-valued modular forms, we have $F\mapsto F_0$, while in the other direction there is also a canonical lift. Unfortunately, neither of them are injective or surjective. By imposing Hecke's $\epsilon$-condition on level $p$ forms with character $\left(\frac{\cdot}{p}\right)$ (see \cite{hecke1940analytische}) and applying an idea of Krieg  \cite{krieg1991maass}, Bruinier and Bundschuh \cite{bruinier2003borcherds} established a clean one-to-one correspondence in this case, which explains well why these modular forms with $\epsilon$-condition behave like modular forms on the full modular group. They then translated Borcherds's theory on obstructions and automorphic products to scalar-valued modular forms. More recently, by carefully investigating Weil representations, Scheithauer \cite{scheithauer2011some} proved that in the case of square-free level, all vector-valued modular forms that are invariant under the automorphisms of the discriminant form are lifts of some scalar-valued modular forms. In other words, he determined the image of the lift from scalar-valued modular forms to vector-valued modular forms in the case when the level is square-free. Note that the invariance condition in the prime level case \cite{bruinier2003borcherds} is hidden in their assumption on the signature of the lattice and the weight of the modular forms.

In the present note, by generalizing Bruinier and Bundschuh's idea in \cite{bruinier2003borcherds} and applying Scheithauer's formulas for the Weil representations in \cite{scheithauer2009weil}, we establish such a correspondence in the case when the discriminant form is given by the norm form of a real quadratic field. In particular, the level $N$ is a positive fundamental discriminant. This is Theorem \ref{Correspondence}, the main result of this note. We then translate the obstruction theorem of Borcherds (Theorem \ref{Obstruction-Scalar}),  and apply it to the case of $\mathbb Q(\sqrt{3})$, that is, the case of level $12$. In particular, using $\eta$-quotients, we construct the unique weakly holomorphic modular form $f_1$ of level $12$, weight $0$, with character $\left(\frac{12}{\cdot}\right)$ and the $\epsilon$-condition, whose Fourier expansion at $\infty$ begins with $q^{-1}$.

Here is the layout of this note. We recall scalar-valued modular forms and vector-valued modular form for Weil representations in Section 1 and 2 respectively. In Section 3, we establish the isomorphism and also prove some properties of the Fourier coefficients at different cusps. In Section 4, with the correspondence, we restate Borcherds's obstruction theorem for scalar-valued modular forms, and we also consider rationality of Fourier coefficients for scalar-valued modular forms at the end. In the last section, we apply the results in previous sections to give an example in the case of level $12$.

We remark that although in this note we are only interested in the case when the discriminant form is obtained from real quadratic fields, many statements should also hold in a more general setting.

\subsection*{Acknowledgments} The author would like to thank Professor Henry H. Kim and Professor Kyu-Hwan Lee for suggesting this problem and for useful communications. The author is also grateful to Professor Jan H. Bruinier for his correspondence with Professor Henry H. Kim regarding this problem. Finally, the author is very thankful to the anonymous referee for carefully reading a previous version of this note and making many valuable comments and suggestions.

\section{Scalar-Valued Modular Forms}\label{scalar}
\noindent
We consider scalar-valued modular forms in this section and recall some operators on spaces of scalar-valued modular forms.

For any positive integer $m$, we denote by $\omega(m)$ the number of distinct prime divisors of $m$. For any pair $m,N$ of integers, we denote by $(m,N)$ the greatest common divisor of $m$ and $N$, which should not be confused with the bilinear form we introduce below. If $N>0$, we denote $N_m$ to be the $m$-part of $N$; that is, $N_m$ is a positive divisor of $N$, contains only primes that divide $m$, and $(N/N_m,m)=1$.

Given a Dirichlet character $\chi$ modulo $N$, we denote $A(N,k,\chi)$ the space of weakly holomorphic modular functions of level $N$, weight $k$ and character $\chi$; namely, the space of functions $f$ that are holomorphic on the upper half plane, meromorphic at cusps, and
\[(f|_kM)(\tau)=\chi(d)f(\tau),\quad\text{for all } M=\begin{pmatrix}a&b\\c&d\end{pmatrix}\in\Gamma_0(N).\] Let $M(N,k,\chi)$ and $S(N,k,\chi)$ be the subspace of holomorphic forms and that of cuspforms respectively.

We are only interested in the case when $N>1$ is a fundamental discriminant and $\chi=\chi_D=\left(\frac N \cdot\right)$. It follows that $\chi_D$ is primitive of modulus $N$. Decompose it into $p$-components as $\chi_D=\prod_p\chi_p$. Then if $p$ is odd, then $\chi_p=\left(\frac \cdot p\right)$, and define $\varepsilon_p=1$ if $p\equiv 1\imod 4$ and $\varepsilon_p=i$ if $p\equiv 3\imod 4$. $\chi_2$ is determined and $\varepsilon_2$ is defined as follows:
\begin{itemize}
\item if $N_1\equiv 1\imod 4$, $\chi_2=1$ and $\varepsilon_2=1$.
\item if $N_1\equiv 3\imod 4$, $\chi_2=\left(\frac{-4}{\cdot}\right)$ and $\varepsilon_2=i$.
\item if $N_1\equiv 2\imod 8$, $\chi_2=\left(\frac{2}{\cdot}\right)$ and $\varepsilon_2=1$.
\item if $N_1\equiv 6\imod 8$, $\chi_2=\left(\frac{-2}{\cdot}\right)$ and $\varepsilon_2=i$.
\end{itemize}
Let $W(\chi)$ denote the Gauss sum of a Dirichlet character $\chi$, that is
\[W(\chi)=\sum_{a\imod N}\chi(a)e^{2\pi i a/N}.\] For any prime divisor $p$ of $N$, we have $W(\chi_p)=\varepsilon_pN_p^{\frac{1}{2}}$.

For each positive divisor $m$ of $N$, we shall denote $\chi_m=\prod_{p\mid m}\chi_p$ and $\chi_m'=\prod_{p\mid N,p\nmid m}\chi_p$.

For convenience, we denote the matrices
\[
S=\begin{pmatrix}
0&-1\\
1&0
\end{pmatrix},\quad T=\begin{pmatrix}
1&1\\
0&1
\end{pmatrix},\quad I=\begin{pmatrix}
1&0\\
0&1
\end{pmatrix},\quad W(N)=\begin{pmatrix}
0&-1\\
N&0
\end{pmatrix}.
\]
The weight-$k$ slash operator on a function $f$ on the upper half plane is defined as
\[(f|_kM)(\tau)=(\text{det}M)^{\frac{k}{2}}(c\tau+d)^{-k}f(M\tau), \text{ for } M=\begin{pmatrix}
a&b\\
c&d
\end{pmatrix}\in \text{GL}_2^+(\mathbb R).\] It follows that $W(N)$ gives an involution on $A(N,k,\chi_D)$. If $f\in A(N,k,\chi_D)$ and $m\mid N$, then the Hecke operator $U(m)$ is defined as
\[(f|_kU(m))(\tau)=m^{\frac{k}{2}-1}\sum_{j\imod m}f\left|_k\begin{pmatrix}
1&j\\
0&m
\end{pmatrix}\right.\]

For a positive divisor $m$ of $N$, choose $\gamma_m\in\text{SL}_2(\mathbb Z)$ such that
\[\gamma_m\equiv
\left\{
\begin{array}{cl}
S&\imod (N_m)^2\\
I &\imod (N/N_m)^2
\end{array}
\right.,
\] and define $\eta_m=\gamma_m\begin{pmatrix}
N_m&0\\
0&1
\end{pmatrix}$ and denote $\eta_m'=\eta_{N/N_m}$.

\begin{Lem}\label{Eta-operator} Let $f\in A(N,k,\chi_D)$ and $m,m_1,m_2$ be positive divisors of $N$.

(1) The action $f|_k\eta_m$ is independent of the choice of $\gamma_m$ and it defines an operator on $A(N,k,\chi_D)$.

(2) $f|_k\eta_N=f|_kW(N)$.

(3) If $(m_1,m_2)=1$, $f|_k\eta_{m_1m_2}=\chi_{m_2}(N_{m_1})f|_k\eta_{m_1}\eta_{m_2}$. In particular, $f|_k\eta_m\eta_m'=\chi_m'(N_m)f|_kW(N)$. Moreover, if $m=p_1p_2\cdots p_k$ is square-free, then \[f|_k\eta_m=\prod_{i<j}\chi_{p_j}(N_{p_i})f|_k\eta_{p_1}\eta_{p_2}\cdots\eta_{p_k}.\]

(4) $f|_k\eta^2_m=\chi_m(-1)\chi_m'(N_m)f$.

(5) If $(m_1,m_2)=1$, $f|_k\eta_{m_1}U(N_{m_2})=\chi_{m_1}(N_{m_2})f|_kU(N_{m_2})\eta_{m_1}$.
\end{Lem}
\begin{proof}
In Section 4.6 of Miyake's book \cite{miyake2006modular}, the case when $m=p$ is a prime is treated. The general case follows from the same verifications and we skip the details.
\end{proof}

From now on, we shall drop the weight in the notations of the operators if no confusion is possible.

\section{Modular Forms for the Weil Representaions}\label{vector}

Let $N_1>1$ be a square-free integer. Let $F=\mathbb Q(\sqrt{N_1})$ and $\mathcal O_F$ be its ring of integers. Let $N$ be the discriminant of $F/\mathbb Q$. It is well-known that if $N_1\equiv 2,3\imod 4$, $N=4N_1$ and $\mathcal O_F=\mathbb Z[\sqrt{N_1}]$, and if $N_1\equiv 1\imod 4$, $N=N_1$ and $\mathcal O_F=\mathbb Z\left[\frac{\sqrt{N_1}+1}{2}\right]$. Let N and Tr denote the norm and trace for $F/\mathbb Q$ respectively.
If $\mathfrak d$ is the different of $F/\mathbb Q$, we know that
\[\mathfrak d^{-1}=\{x\in F: \text{Tr}(x\mathcal O_F)\subset \mathbb Z\}=
\left\{
\begin{array}{cl}
\frac{1}{2}\mathbb Z+\frac{\sqrt{N_1}}{2N_1}\mathbb Z, & \quad N_1\equiv 2,3\imod 4\\
\left(\frac{1}{\sqrt{N_1}}\right), &\quad N_1\equiv 1\imod 4
\end{array}
\right.
\]
Define the following lattice $L=\mathbb Z^2\oplus \mathcal O_F$ with the quadratic form
\[q(a,b,\gamma)=\text{N}(\gamma)-ab,\quad a,b\in\mathbb Z, \gamma\in\mathcal O_F.\]
The corresponding bilinear form is given by
\[((a_1,b_1,\gamma_1),(a_2,b_2,\gamma_2))=\text{Tr}(\gamma_1\overline{\gamma_2})-a_1b_2-a_2b_1.\] We see that $L$ is an even lattice of signature $(2,2)$. Its dual lattice is $L'=\mathbb Z^2\oplus \mathfrak d^{-1}$, hence the discriminant form $D=L'/L\cong \mathfrak d^{-1}/\mathcal{O}_F$. The level of $D$ is $N$. Denote $q\imod 1$ on $D$ also by $q$.

Let $k$ be an even integer. Let $\rho_D$ be the Weil representation of $SL_2(\mathbb Z)$ on $\mathbb C[D]$; that is, if $\{e_\gamma:\gamma\in D\}$ is the standard basis for the group algebra $\mathbb C[D]$, then the action
\begin{align*}
\rho_D(T)e_\gamma &= e(q(\gamma))e_\gamma,\\
\rho_D(S)e_\gamma &=\frac{1}{\sqrt{N}}\sum_{\delta\in D}e(-(\gamma,\delta))e_\delta,
\end{align*}
defines the unitary representation $\rho_D$ of $\text{SL}_2(\mathbb Z)$ on $\mathbb C[D]$. Here $e(x)=e^{2\pi i x}$ and $T,S$ are the standard generators of $\text{SL}_2(\mathbb Z)$(see the next section).

Let $\mathcal A(k,\rho_D)$ be the space of modular forms of weight $k$ and type $\rho_D$. That is, $F=\sum_\gamma F_\gamma e_\gamma\in\mathcal A(k,\rho_D)$ if $F|_kM:=\sum_\gamma (F_\gamma|_kM)e_\gamma=\rho_D(M)F$ for any $M\in\text{SL}_2(\mathbb Z)$, $F_\gamma$ is holomorphic on the upper half plane and $F_\gamma=\sum_{n\in q(\gamma)+\mathbb Z}a(\gamma,n)q^n$ with at most finitely many negative power terms. Let $\mathcal M(k,\rho_D)$ and $\mathcal S(k,\rho_D)$ denote the space of holomorphic forms and the space of cusp forms respectively.
We shall also need $\mathcal A^\text{inv}(k,\rho_D)$, the subspace of modular forms that are invariant under $\text{Aut}(D)$. Analogously we have $\mathcal M^{\text{inv}}(k,\rho_D)$ and $\mathcal S^{\text{inv}}(k,\rho_D)$.

We finish this section by investigating the discriminant forms $D$ considered above. Here we follow the notations in \cite{scheithauer2009weil}.

\begin{Lem}\label{Discriminant-Form}
Denote $q=\prod_{p\mid N}q_p$ be a Jordan decomposition of the discriminant form $q$ on $D$. Then if $p\mid N$ and $p$ is odd, then $q_p\cong p^{\pm 1}$ with $\pm 1=\left(\frac{-2N_1/p}{p}\right)$. For $q_2$,

(1) If $N_1\equiv 1\imod 4$, $q_2$ is trivial.

(2) If $N_1\equiv 3\imod 4$, $q_2\cong 2_2^{+ 2}$.

(3) If $N_1\equiv 2\imod 4$, $q_2\cong 2_1^{+ 1}\oplus 4_t^{\pm 1}$. Here $t=-N_1/2$ and $\pm 1=\left(\frac{2}{t}\right)$.
\end{Lem}
\begin{proof}
Assume $N_1\equiv 1\imod 4$; clearly $q_2$ is trivial. We know that $L'/L\cong \mathbb Z/N_1\mathbb Z$ with a generator $\frac{1}{\sqrt{N_1}}$, hence a generator for the p-Jordan component can be chosen as $\frac{N_1/p}{\sqrt{N_1}}=\frac{\sqrt{N_1}}{p}$. We have
$q_p\left(\frac{\sqrt{N_1}}{p}\right)=\frac{-N_1/p}{p}$, from which it follows that $q_p\cong p^{\pm 1}$ with $\pm 1=\left(\frac{-2N_1/p}{p}\right)$.

Now assume $N_1\equiv 3\imod 4$. In this case, $L'/L\cong \mathbb Z/2\mathbb Z\times \mathbb Z/2\mathbb Z\times \mathbb Z/N_1\mathbb Z$, with generators $\gamma_2=\frac{1}{2}$, $\gamma_2'=\frac{\sqrt{N_1}}{2}$ and $\gamma_{N_1}=\frac{\sqrt{N_1}}{N_1}$ for each component respectively. From this, it follows that $q_2\cong 2_2^{+2}$. For any prime $p\mid N_1$, we may choose $\frac{\sqrt{N_1}}{p}$ as a generator for the p-Jordan component and then we see that $q_p\cong p^{\pm 1}$ with $\pm 1=\left(\frac{-2N_1/p}{p}\right)$.

Finally assume $N_1\equiv 2\imod 4$. In this case, $L'/L\cong \mathbb Z/2\mathbb Z\times \mathbb Z/4\mathbb Z\times \mathbb Z/(N_1/2)\mathbb Z$, with generators $\gamma_2=\frac{1}{2}$, $\gamma_2'=\frac{\sqrt{N_1}}{4}$ and $\gamma_{N_1/2}=\frac{2\sqrt{N_1}}{N_1}$ for each component respectively. Since $q(\gamma_2')=-\frac{N_1}{16}=-\frac{N_1/2}{8}$, we see that $q_2\cong 2_1^{+1}\oplus 4_{-N_1/2}^{\pm 1}$, where $\pm 1=\left(\frac{2}{-N_1/2}\right)$. For any odd prime $p\mid N_1$, we may choose $\frac{\sqrt{N_1}}{p}$ as a generator for the p-Jordan component and then we see that $q_p\cong p^{\pm 1}$ with $\pm 1=\left(\frac{-2N_1/p}{p}\right)$.
\end{proof}

Let $D=\bigoplus_{p\mid N}D_p$ be a Jordan decomposition as in the proof of Lemma \ref{Discriminant-Form}. In the case of $N_1\equiv 2,3\imod 4$, we keep the generators $\gamma_2,\gamma_2'$ therein.

\begin{Lem} The group $\text{Aut}(D)$ is an elementary abelian 2-group of order $2^{\omega(N)}$. More explicitly,

(1) If $p\mid N$ is odd, then $\text{Aut}(D_p)=\langle \sigma_p\rangle$, with $\sigma_p(\gamma)=-\gamma$, $\gamma\in D_p$.

(2) If $N_1\equiv 3\imod 4$, then $\text{Aut}(D_2)=\langle \sigma_2\rangle$, with  $\sigma_2(\gamma_2)=\gamma_2'$.

(3) If $N_1\equiv 2\imod 4$, then $\text{Aut}(D_2)=\langle \sigma_2\rangle$, with $\sigma_2(\gamma)=-\gamma$, $\gamma\in D_2$.
\end{Lem}
\begin{proof}
This can be seen by explicit computations. We omit the details.
\end{proof}

\begin{Prop}\label{Invariance}
If $\beta,\gamma\in D$ with $q(\beta)=q(\gamma)$, then there exists $\sigma\in\text{Aut}(D)$ such that $\sigma\beta=\gamma$.
\end{Prop}
\begin{proof}
Assume $D=D_1\oplus D_2$ with $(|D_1|, |D_2|)=1$. Let $\gamma=\gamma_1+\gamma_2$ and $\beta=\beta_1+\beta_2$ with $\gamma_i,\beta_i\in D_i$ for $i=1,2$. Assume that $q(\gamma)=q(\beta)$. We see easily that $q(\gamma_i)=q(\beta_i)$ for each $i=1,2$. So to finish the proof, we just need to consider each Jordan component.

Let $p\mid N$ be any odd prime and $\gamma,\beta$ belong to the p-Jordan component which is isomorphic to $\mathbb Z/p\mathbb Z$. Since $q(\gamma)=q(\beta)$, we must have $\gamma=\pm \beta$. So either identity or $\sigma_p$ can do the job.

To consider the 2-Jordan components, we keep the notations in the proof of Lemma \ref{Discriminant-Form}. We first assume that $N_1\equiv 3\imod 4$. In this case, we have either $\gamma=\beta$ or $\{\gamma,\beta\}=\{\gamma_2,\gamma_2'\}$. So either identity or $\sigma_2$ can send $\gamma$ to $\beta$.

Similarly, if $N_1\equiv 2\imod 4$, we see that either $\gamma=\beta$ or $\{\gamma,\beta\}=\{\gamma_2',3\gamma_2'\}$ or $\{\gamma,\beta\}=\{\gamma_2+\gamma_2',\gamma_2+3\gamma_2'\}$. In the latter two cases, $\sigma_2$ sends $\gamma$ to $\beta$.
\end{proof}

 Define $\epsilon=(\epsilon_p),\epsilon^*=(\epsilon_p^*)\in\{\pm 1\}^{\omega(N)}$ such that $\epsilon_p^*=1$ for each prime $p\mid N$, and
$\epsilon_p=\chi_p(-1)$ if $p$ is odd,
$\epsilon_2=-1$ if $N_1\equiv 3\imod 4$, and
$\epsilon_2=\chi_{N_1/2}(-1)$ if $N_1\equiv 2\imod 4$. We shall need these data in building the correspondence. Here $\epsilon$ is not to be confused with the $\varepsilon$ defined in Section 1.

\section{The Isomorphism Theorem}
\noindent
In this section, we prove an isomorphism between some spaces of vector-valued modular forms and scalar-valued modular forms.

Fix a modular form $F\in\mathcal A^\text{inv}(k,\rho_D)$. Define $W$ the span of $F_\gamma$, $\gamma\in D$, and $W'$ the span of $F_0|M$, $M\in SL_2(\mathbb Z)$. Let $W_0$ be the subspace of $T$-invariant functions in $W$.

\begin{Lem}\label{Term-Inclusion}
Let $S\subset D$. If $\sum_{\gamma\in S}F_\gamma\in W'$, then $F_\gamma\in W'$ for any $\gamma\in S$.
\end{Lem}
\begin{proof}
We may write $\sum_{\gamma\in S}F_\gamma=\sum_{n\imod \mathbb Z}F_n$, with $F_n=\sum_{\gamma\in S, q(\gamma)=n}F_\gamma$. Since $F$ is invariant under $\text{Aut}(D)$, by Proposition \ref{Invariance}, terms in the sum of $F_n$ are all equal. Therefore, we only have to prove that $F_n\in W'$.

Now the transformation rule of $F$ under $T$ shows that $F_n|T=e(n)F_n$; here $e(x)=e^{2\pi ix}$. Since $W'$ is invariant under the action of $\text{SL}_2(\mathbb Z)$ and $\sum_{n}F_n\in W'$, we have $\sum_ne(nj)F_n\in W'$ for each positive integer $j$. Since $e(n)$'s are distinct mutually, this implies that $F_n\in W'$ by the theory of Vandermonde matrix.
\end{proof}

\begin{Lem}\label{T-invariance}
$W_0=\text{span}_\mathbb{C}\{F_0\}$. Actually, if $f=\sum_{\gamma\in D}a_\gamma F_\gamma\in W_0$, then $f=a_0F_0$.
\end{Lem}
\begin{proof}
The lemma follows easily from the fact that $D$ has no nonzero isotropic elements.
\end{proof}

\begin{Prop}
$W=W'$.
\end{Prop}
\begin{proof}
This is done in \cite{scheithauer2011some} in the case of square-free $N$. We now consider the case $N_1\equiv 3\imod 4$ and $N=4N_1$. One direction is trivial.

For $M\in\text{SL}_2(\mathbb Z)$, we shall need the concrete formula for $F_0|M$ in \cite{scheithauer2011some}:
\[F_0|M=(*) \sum_{\beta\in D^{c*}}e(d\beta_c^2/2)F_\beta,\quad M=\begin{pmatrix}a&b\\c&d\end{pmatrix}.\]
Here $(*)$ stands for a nonzero constant and we refer to \cite{scheithauer2011some} for the meaning of other notations.
Let $m_1$ be a positive divisor of $N_1$.  We fix $\gamma_p$ a generator for the p-adic Jordan component if $p>2$ and $\gamma_2$ and $\gamma_2'$ a set of generators for the 2-adic Jordan component. Therefore, $D_p=\langle \gamma_p\rangle$ if $p>2$ and $D_2=\langle \gamma_2\rangle \times \langle \gamma_2'\rangle$.

We choose $M\in\text{SL}_2(\mathbb Z)$ such that $c=1$, $d=N$. In this case, $D_c=0$, $D^c=D$, $x_c=0$ and $D^{c*}=D^c$. It follows that $d\beta_c^2/2=0$ and $\sum_{\beta\in D}F_\beta\in W'$. By Lemma \ref{Term-Inclusion}, we have $F_\gamma\in W'$ for each $\gamma\in D$, hence $W=W'$.

Similar argument can be applied to the case $N_1\equiv 2\imod 4$. We skip the details.
\end{proof}

Define a map $\phi: \mathcal A^\text{inv}(k,\rho_D)\rightarrow A(N,k,\chi_D)$ by
\[F\mapsto 2^{-\omega(N)}N^{-\frac{k-1}{2}}F_0|{W(N)},\]
where $\omega(N)$ is the numbers distinct prime divisors.
Define another map $\psi: A(N,k,\chi_D)\rightarrow \mathcal A^\text{inv}(k,\rho_D)$ by
\[f\mapsto N^{\frac{k-1}{2}}\sum_{M\in\Gamma_0(N)\backslash SL_2(\mathbb Z)}\left(f|W(N)|M\right)\rho_D(M^{-1})e_0.\]

\begin{Lem} Both $\phi$ and $\psi$ are well-defined.
\end{Lem}
\begin{proof}
To consider either one, we may drop the operator $W(N)$. It is easy to verify that $\psi$ is well-defined and that for $\phi$ follows from Proposition 4.5 in \cite{scheithauer2009weil} applied to $\gamma=0$.
\end{proof}

\begin{Prop}\label{Injection}
We have $\psi\circ\phi=id$. In particular, $\phi$ is injective.
\end{Prop}
\begin{proof}
When $N_1\equiv 1\imod 4$, this is proved in Theorem 5.4 in \cite{scheithauer2011some}. Here we prove the case $N_1\equiv 3\imod 4$ and omit the case $N_1\equiv 2\imod 4$, since the latter is similar.

We need to prove that if $F\in\mathcal A^\text{inv}(k,\rho_D)$, then
\[\sum_{M\in\Gamma_0(N)\backslash\text{SL}_2(\mathbb Z)} (F_0|M)(\rho_D(M^{-1})e_0,e_0)=2^{\omega(N)}F_0.\]
The inequivalent cusps are represented by $\frac{1}{m_1}$, $\frac{1}{2m_1}$, and $\frac{1}{4m_1}$, where $m_1$ runs over the set of all positive divisors of $N_1$.
For each cusp $s$, consider
\[F_s=\sum_{M\in\Gamma_0(N)\backslash\text{SL}_2(\mathbb Z), M\infty=s}F_0|M(\rho_D(M^{-1})e_0,e_0),\]
so we have to prove $\sum_{s}F_s=2^{\omega(N)}F_0$.

Assume $s\sim \frac{1}{m_1}$ and choose $M$ with $c=m_1$. Then since $F_s$ is $T$-invariant, that is $F_s(\tau+1)=F_s$, by Lemma \ref{T-invariance} above and Theorem 4.7 in \cite{scheithauer2009weil},
\begin{align*}
F_s &= \sum_{j\imod N/m_1}(F_0|MT^j)(\rho_D(MT^n)^{-1}e_0,e_0)\\
&= \sum_{j\imod N/m_1} \sum_{\alpha\in D} F_\alpha (\rho_D(MT^j)e_\alpha,e_0)(\rho_D(MT^j)^{-1}e_0,e_0)\\
&= \frac{N}{m_1}(\rho_D(M)e_0,e_0)(\rho_D(M)^{-1}e_0,e_0)F_0=F_0.
\end{align*}

If $s\sim \frac{1}{4m_1}$, by similar computations, we have $F_s=F_0$. But if $s\sim \frac{1}{2m_1}$, we have $F_s=0$ since in this case $0\not\in D^{c*}$ with $c=2m_1$.
There are in total $2^{\omega(N)}$ cusps not of the form $\frac{1}{2m_1}$, and the statement follows.
\end{proof}

Now we try to find the image of $\phi$ and establish a one-to-one correspondence between spaces of vector-valued modular forms and scalar-valued modular forms.

Define the subspace $A^\delta(N,k,\chi_D)$ for each $\delta=(\delta_p)_{p\mid N}\in\{\pm 1\}^{\omega(N)}$ of $A(N,k,\chi_D)$ as follows:
\[A^\delta(N,k,\chi_D)=\left\{f=\sum_n a(n)q^n\in A(N,k,\chi_D)\left|
 \begin{split}a(n)=0 \text{ if } \chi_p(n)=-\delta_p \text{ for some }p\mid N\end{split}\right.\right\}.\]
Denote $M^\delta(N,k,\chi_D)=M(N,k,\chi_D)\cap A^\delta(N,k,\chi_D)$ and similarly we have $S^\delta(N,k,\chi_D)$.
We call such a condition we impose on the Fourier coefficients the \emph{$\delta$-condition}. Hence $\epsilon$ and $\epsilon^*$ are two special sign vectors.

\begin{Lem}
$\phi(\mathcal A^\text{inv}(k,\rho_D))\subset A^\epsilon(N,k,\chi_D)$.
\end{Lem}
\begin{proof}
For any $F\in\mathcal A^\text{inv}(k,\rho_D)$, we show that $F_0|W(N)\in A^\epsilon(N,k,\chi_D)$.
First of all,
\[W(N)=\begin{pmatrix}0&-1\\1&0\end{pmatrix}\begin{pmatrix}N&0\\0&1\end{pmatrix},\] so
\[F_0|W(N)=F_0\left|S\begin{pmatrix}N&0\\0&1\end{pmatrix}\right.=\frac{1}{\sqrt{N}}\sum_{\gamma\in D}F_\gamma\left|\begin{pmatrix}N&0\\0&1\end{pmatrix}\right.=N^{\frac{k-1}{2}}\sum_{\gamma\in D}F_\gamma(N\tau).\]
Assume that $F_\gamma(\tau)=\sum_{n\in\mathbb Z+q(\gamma)}a(\gamma,n)q^n$. Then we have
\[F_0|W(N)=N^{\frac{k-1}{2}}\sum_{n\in\mathbb Z}\left(\sum_{\gamma\in D: \frac{n}{N}= q(\gamma)}a\left(\gamma, nN^{-1}\right)\right)q^n.\] Let $a(n)$ be the Fourier coefficient of $F_0|W(N)$, we are supposed to show that $a(n)=0$ if $(N,n)=1$ and $\chi_p(n)=-\epsilon_p$ for some $p\mid N$.

Assume $N_1\equiv 1\imod 4$. If $\chi_p(n)=-\epsilon_p$, then $\left(\frac n p\right)=-\left(\frac{p^{-1}N\alpha_p}{p}\right)$, so $q_p$ does not represent $\frac{p^{-1}nN}{p}$. Hence $q$ does not represent $\frac{n}{N}$, that is, there is no $\gamma\in D$ such that $q(\gamma)=\frac{n}{N}$. Therefore, $a(n)=0$.

Assume now $N_1\equiv 3\imod 4$. For the odd prime factors, we argue as above. In case of the prime $2$, if $\left(\frac{-1}{n}\right)=1$, then $n\equiv 1\imod 4$. Since $q_2\cong 2_2^{+2}$, there exists no element with $q_2$-norm $\frac{3}{4}$. Therefore, there is no element with $q$-norm $\frac{n}{N}$, since other wise there would be an element $\gamma$ with $N_1q(\gamma)=\frac{nN_1}{N}=\frac{n}{4}$ and hence an element $\gamma_2$ with $q(\gamma_2)=\frac{3}{4}$.

Assume $N_1\equiv 2\imod 4$. Similarly, we only have to consider the $2$-component. We first deal with the case $N_1\equiv 2\imod 16$. Now $\chi_2(n)=\left(\frac 2 n\right)$. That $\chi_2(n)=-\epsilon_2$ means $n\equiv \pm 3\imod 8$. If there exists $\gamma$ such that $q(\gamma)=\frac{n}{N}$ hence $\frac{N_1}{2}q(\gamma)=\frac{n}{8}$, then there exists $\gamma_2$ such that $q(\gamma_2)=\pm \frac{3}{8}$. But explicit computations show that the local possible norms at $2$ are
\[0,\quad -\frac{1}{8}, \quad \frac{1}{2}, \quad \frac{1}{4}, \quad \frac{1}{8}, \quad \frac{3}{4}.\] So there are no elements with norms $\pm\frac{3}{8}$ and it follows that $a(n)=0$.
The other cases when $N_1\equiv 2\imod 4$ follow similarly and we omit the details.
\end{proof}

Define the operator $J(p,\delta_p): A(N,k,\chi_D)\rightarrow A(N,k,\chi_D)$ for each $\delta_p\in\{\pm 1\}$ by
\[f|J(p,\delta_p)=\frac{1}{2}\left(f+\delta_pC_p f|U(N_p)|\eta_p\right).\]
where $C_p=N_p^{-\frac{k-2}{2}}W(\chi_p)^{-1}\chi_p(-1)$. Here $W(\chi)$ is the Gauss sum of $\chi$.

\begin{Lem}\label{Decomposition-Lem}
Assume $f|J(p,\delta_p)=\sum_n a(n)q^n$. Then $a(n)=0$ if $\chi_p(n)=-\delta_p$. Moreover, $f=f|J(p,+1)+f|J(p,-1)$.
\end{Lem}
\begin{proof}
Choose $a,b\in\mathbb Z$ such that $(a,p)=1$ and $ab+1\equiv 0\imod N_p$ and $a\equiv b\imod N/N_p$. Let
\[\gamma=\begin{pmatrix}1&a\\0&N_p\end{pmatrix}\gamma_p\begin{pmatrix}1&b\\0&N_p\end{pmatrix}^{-1}.\]
Hence
$\gamma\in \Gamma_0(N)$, $\chi_D(\gamma)=\chi_p(-b)$, and
\[f\left|\begin{pmatrix}1&a\\0&N_p\end{pmatrix}\gamma_p\right.=\chi_p(-b)f\left|\begin{pmatrix}1&b\\0&N_p\end{pmatrix}\right.
.\]
Take the summation $\sum'$ with $a$, hence $b$, over $\left(\mathbb Z/N_p\mathbb Z\right)^\times$, and we have
\[{\sum_a}' f\left|\begin{pmatrix}1&a\\0&N_p\end{pmatrix}\eta_p\right.=W(\chi_p)\chi_p(-1)
\sum_{n}\chi_p(n)a(n)q^n,\]
where $W(\chi_p)$ is the Gauss sum and we used the fact that $\chi_p$ is primitive modulo $N_p$. On the other hand, assuming $N_p=p^e$,
\[N_p^{\frac{k}{2}-1}{\sum_{a}}'f\left|\begin{pmatrix}1&a\\0&N_p\end{pmatrix}\right.
=f|U(p^e)-p^{\frac{k}{2}-1}f\left|U(p^{e-1})\begin{pmatrix}1&0\\0&p\end{pmatrix}\right.,\] which implies
\begin{equation}\label{J-operator}
f|J(p,\delta_p)=\frac{1}{2}\sum_n\left(1+\delta_p\chi_p(n)\right)a(n)q^n+\frac{1}{2}\delta_pC_pp^{\frac{k-2}{2}}
f\left|U(p^{e-1})\begin{pmatrix}1&0\\0&p\end{pmatrix}\eta_p\right..\end{equation}
It is easy to check that
\[\begin{pmatrix}1&0\\0&p\end{pmatrix}\eta_p\begin{pmatrix}p&0\\0&1\end{pmatrix}^{-1}\eta_p^{-1}\equiv \left\{\begin{matrix}I& \imod N_p\\ \begin{pmatrix}p^{-1}&0\\ 0&p\end{pmatrix}& \imod N/N_p\end{matrix}\right.\] It follows that in the second term of the above expression for $f|J(p,\delta_p)$, if the $q^n$-Fourier coefficient is not zero, then $p\mid n$.
Now it is clear that if $\chi_p(n)=-\delta_p$, then the $q^n$-Fourier coefficients in both the two terms are zero, and this finishes the proof of the first statement.

Suppose $b(n)$ is the n-th Fourier coefficient of $f-f|J(p,+1)-f|J(p,-1)$. It is obvious from the expression above that $b(n)=0$ if $p\nmid n$ and then it must be $0$ since $\chi_D$ is primitive (see, for example, the proof of Theorem 4.6.4 in \cite{miyake2006modular}).
\end{proof}

We prove a lemma before we prove the decomposition of the whole space into subspaces with $\delta$-conditions.

\begin{Lem}\label{Commute}
Let $p,q$ be two distinct prime divisors of $N$ and $\delta_p,\delta_q\in\{\pm 1\}$. Then
$J(p,\delta_p)$ and $J(q,\delta_q)$ commute.
\end{Lem}
\begin{proof}
It is not hard to see that we need to prove
\[\chi_p(N_q)f|\eta_p\eta_q=\chi_q(N_p)f|\eta_q\eta_p.\]
But by Lemma \ref{Eta-operator}, both sides are equal to $f|\eta_{pq}$. This finishes the proof.
\end{proof}

\begin{Prop}\label{Decomposition} We have
\[A(N,k,\chi_D)=\bigoplus_\delta A^\delta(N,k,\chi_D),\]
where $\delta$ runs over $\{\pm 1\}^{\omega(N)}$.
\end{Prop}
\begin{proof}
Suppose $\sum_\delta f_\delta=0$ with $f_\delta\in A^\delta(N,k,\chi)$. Separate the sum for each $p\mid N$ as follows:
\[f_p:=\sum_{\delta:\delta_p=1}f_\delta=-\sum_{\delta:\delta_p=-1}f_\delta.\]
From the definition of $A^\delta(N,k,\chi_D)$, $f_p$ is a modular form with Fourier coefficients $a(n)$ such that $a(n)=0$ if $p\nmid n$. Since the conductor of $\chi_D$ is $N$, it follows that $f_p=0$ (see, for example, the proof of Theorem 4.6.4 in \cite{miyake2006modular}). By considering all $p\mid N$ in the same way, we see that $f_\delta=0$ for any $\delta$.

We still need to prove that
\[A(N,k,\chi_D)=\sum_\delta A^\delta(N,k,\chi_D).\]
This can be seen from the previous two lemmas by successively applying the $J(p,+1)$ and $J(p,-1)$ until we exhaust all $p\mid N$. Indeed, for each $\delta$, assuming $f=\sum_n a(n)q^n$, we consider
\[g=f|J(p_1,\delta_{p_1})J(p_2,\delta_{p_2})\cdots J(p_l,\delta_{p_l})=\sum_nb(n)q^n,\]
where $N=N_{p_1}N_{p_2}\cdots N_{p_l}$ and $l=\omega(N)$. By Lemma \ref{Decomposition-Lem}, we see that if $\chi_{p_l}(n)=-\delta_p$, then $b(n)=0$. Then by Lemma \ref{Commute}, the same holds for any $p_i$, $1\leq i\leq l$, hence $g\in A^\delta(N,k,\chi_D)$.
In the end, we have a decomposition of $2^{\omega(N)}$ terms for $f$, each term of which belongs to $A^\delta(N,k,\chi_D)$ for a unique $\delta$.
\end{proof}

\begin{Lem} \label{Projection} Let $p\mid N$ be a prime and let $\delta_p\in\{\pm 1\}$.
Let $f\in A(N,k,\chi_D)$ with $f(\tau)=\sum_{n}a(n)q^n$ be such that $a(n)=0$ if $\chi_p(n)=-\delta_p$. Then $f|J(p,\delta_p)=f$ and $f|J(p,-\delta_p)=0$. In particular, $J(p,\delta_p)^2=J(p,\delta_p)$ and $J(p,\delta_p)J(p,-\delta_p)=0$.
\end{Lem}
\begin{proof} Assume $\delta_p'\in\{\pm 1\}$ and
\[f=\sum_n a(n)q^n\quad \text{and}\quad f|J(p,\delta_p')=\sum_nb(n)q^n.\]
Assume $p\nmid n$. Then by Equation (\ref{J-operator}), $b(n)=\frac{1}{2}(1+\delta_p'\chi_p(n))a(n)$.

If $\delta_p'=-\delta_p$, then one of $(1-\delta_p\chi_p(n))$ and $a(n)$ is zero, so $b(n)=0$. This implies $f|J(p,-\delta_p)=0$.
If $\delta_p'=\delta_p$, then $b(n)=a(n)$, which implies $f|J(p,\delta_p)=f$.
\end{proof}

\begin{Cor}\label{Fourier-Epsilon}
Let $f\in A(N,k,\chi_D)$. Then $f\in A^\delta(N,k,\chi_D)$ if and only if \[f=\delta_pC_pf|U(N_p)\eta_p, \quad\text{ for each } p\mid N.\]
\end{Cor}
\begin{proof} By Proposition \ref{Decomposition} and Lemma \ref{Projection}, we know that $f\in A^\delta(N,k,\chi_D)$ if and only if $f|J(p,-\delta_p)=0$ for each $p\mid N$, hence if and only if $f=\delta_pC_pf|U(N_p)\eta_p$ for each $p\mid N$. This finishes the proof.
\end{proof}

\begin{Cor}\label{Holomorphy}
Assume $N_1\equiv 1,3\imod 4$ and let $f$ be a function on the upper half plane that satisfies all properties of a weakly holomorphic modular form except possibly the meromorphy at cusps.
If $f$ is meromorphic (or is holomorphic, or vanishes, respectively) at $\infty$, then $f\in A^\delta(N,k,\chi_D)$ (or $M^\delta(N,k,\chi_D)$, or $S^\delta(N,k,\chi_D)$, respectively).
\end{Cor}
\begin{proof} We note that in the following proof, we write $f=(*)g$ if $f=cg$ for some nonzero constant $c\in\mathbb C$ and at different places the constants are different in general.

Consider first the case $N_1\equiv 1\imod 4$. Given any positive divisor $m_1$ of $N_1$, we have $\gamma_{m_1}\infty\sim\frac{1}{N_1/m_1}$.
Since $\gamma_{m_1}\infty$ exhausts all cusps, we just need to show that $f|\gamma_{m_1}$ has the same (or stronger) holomorphy at $\infty$ as  $f$ does. Actually, by a version of Corollary \ref{Fourier-Epsilon} without the meromorphy condition at the cusps, we see that
\[f|\gamma_{m_1}=(*)f\left|U(m_1)\begin{pmatrix}1&0\\0&m_1\end{pmatrix}\right..\]
This finishes the case $N_1\equiv 1\imod 4$.

Now we consider the case $N_1\equiv 3\imod 4$. Let $m_1$ be any positive divisor of $N_1$. We know that $\gamma_{m_1}\infty\sim \frac{1}{4N_1/m_1}$ and $\gamma_{2m_1}\infty\sim \frac{1}{N_1/m_1}$, so the holomorphy at these cusps follows from the same argument as in the case of $N_1\equiv 1\imod 4$. We still need to consider the cusps $s\sim \frac{1}{2m_1}$. Since it involves many computations, here we only sketch the  idea.

Let $f\in A(N,k,\chi_D)$. Then it can be shown that \[f\left|\eta_{2m_1}\begin{pmatrix}1 &0\\0 &2\end{pmatrix}\right.=\chi_{N_1/m_1}(2)f\left|\begin{pmatrix}2 &0\\0 &1\end{pmatrix}\eta_{2m_1}\right..\] Moreover, if $f=\sum_na(n)q^n$, then by a similar argument as in the proof of Lemma \ref{Decomposition-Lem},
\[f|U(8)\eta_2=-i2^{k-1}\sum_n\chi_2(n)a(2n)q^{n}+2^{\frac{k}{2}-1}\chi_{N_1}(2)f\left|U(4)\eta_2\begin{pmatrix}2 &0\\0 &1\end{pmatrix}\right..\] Let
\[\alpha_{2m_1}=\eta_{2m_1}^{-1}\begin{pmatrix}1 &-1/2\\0 &1\end{pmatrix}\eta_{2m_1}\gamma_{N_1/m_1},\]
then it can be seen easily that $\alpha_{2m_1}\in\text{SL}_2(\mathbb Z)$ and $\alpha_{2m_1}\infty\sim \frac{1}{2m_1}$.

Now let $f\in A^\delta(N,k,\chi_D)$, so we need to consider $f|\alpha_{2m_1}$. Since $f=(*)f|U(4m_1)\eta_{2m_1}$, we have
\[f|\alpha_{2m_1}=(*)f|U(4m_1)\eta_{2m_1}\alpha_{2m_1}=(*)f\left|U(4m_1)\begin{pmatrix}1 &-1/2\\0 &1\end{pmatrix}\eta_{2m_1}\gamma_{N_1/m_1}\right..\] Let $g=f|U(4m_1)$, and it is easy to show that
\[g\left|\begin{pmatrix}1 &-1/2\\0 &1\end{pmatrix}\right.=-g(\tau)+2(g|U(2))(2\tau).\] Therefore,
\begin{eqnarray*}f|\alpha_{2m_1}&=&(*)f|U(4m_1)\eta_{2m_1}\gamma_{N_1/m_1}+(*)f\left|U(4m_1)U(2)\begin{pmatrix}2 &0\\0 &1\end{pmatrix}\eta_{2m_1}\gamma_{N_1/m_1}\right.\\
&=& (*)f|\gamma_{N_1/m_1}+(*)f\left|U(8m_1)\begin{pmatrix}2 &0\\0 &1\end{pmatrix}\eta_{2N_1}\begin{pmatrix}1 &0\\0 &N_1/m_1\end{pmatrix}\right.\\
&=& (*)f|\gamma_{N_1/m_1}+(*)f\left|U(8)\eta_2U(N_1/m_1)\begin{pmatrix}1 &0\\0 &2N_1/m_1\end{pmatrix}\right.,
\end{eqnarray*}
and the holomorphy at the cusp $\frac{1}{2m_1}$ follows from above calculations.
\end{proof}

With little effort, the above argument proves more.
\begin{Cor}\label{Holomorphy-Special} Let $N_1\equiv 1,3\imod 4$.
Assume $f\in A(N,k,\chi_D)$ and its Fourier expansion at $\infty$ contains only negative power terms of the form $a(-m)q^{-m}$ with $(m,N)=1$. If $f\in A^\delta(N,k,\chi_D)$, then $f$ is holomorphic at all other cusps.
\end{Cor}
\begin{proof}
Note that applying the operator $U(m')$ with $m'\mid N$ only collects the Fourier coefficients of $f$ at $\infty$ that are multiples of $m'$. From the condition on the negative power terms of $f$ at $\infty$ and the computations in Corollary \ref{Holomorphy}, the only thing we need to take care of is the cusp $\frac{1}{2N_1}$; that is, $N_1\equiv 3\imod 4$ and $m_1=N_1$. To this end, we just have to keep track of the constant scalars and prove that the two terms in $f|\alpha_{2N_1}$ that are scalar multiples of $f$ actually cancel out.
\end{proof}

\bigskip

We denote again by $\psi$ its restriction to the subspace $A^\epsilon(N,k,\chi_D)$. We are ready to establish the one-to-one correspondence between the space of invariant vector-valued modular forms and that of scalar-valued modular forms with $\epsilon$-condition.

\begin{Prop}\label{Surjection}
We have $\phi\circ\psi=id$.
\end{Prop}
\begin{proof}
Let $f\in A^\epsilon(N,k,\chi_D)$. Define for each cusp $s$,
\[F_s=\sum_{M\in\Gamma_0(N)\backslash SL_2(\mathbb Z), M\infty=s}(f|W(N)|M)\rho_D(M^{-1})e_0.\] What we have to prove is that
\[\sum_{s}(F_s,e_0)|W(N)=2^{\omega(N)}f.\]

We first deal with the case $N_1=N\equiv 1\imod 4$.
For a positive divisor $m_1$ of $N$, consider the cusp $s\sim \frac{1}{N/m_1}$. We have
\begin{align*}
(F_s,e_0)|W(N)&= \sum_{j\imod m_1}f|W(N)\gamma_{m_1}T^jW(N)\left(\rho_D(\gamma_{m_1}^{-1})e_0,e_0\right)\\
&= m_{1}^{-\frac{k}{2}+1}f|W(N)\eta_{m_1}U(m_1)W(N)\left(\rho_D(\gamma_{m_1}^{-1})e_0,e_0\right)\\
&= \left(m_{1}^{-\frac{k}{2}+1}\chi_{m_1}(N/m_1)f|_kU(m_{1})\eta_{m_1}\right)\left(\chi_{m_1}'(-1)m_1^{-\frac{1}{2}}\prod_{p\mid m_1}(\chi_p(N/p)\varepsilon_p)\right)\\
&= \left(\chi_{m_1}(-1)\chi_{m_1}'(-1)\chi_{m_1}(N/m_1)\prod_{p\mid m_1}\chi_p(N/p))\left(\prod_{p,q\mid m_1,p\neq q}\chi_p(q)\right)\right)f=f,
\end{align*}
where we applied Lemma \ref{Eta-operator} and the $\epsilon$-condition of $f$ to $f|_kU(m)\eta_m$, and used Theorem 4.7 in \cite{scheithauer2009weil} for the concrete computation of $\left(\rho_D(\gamma_m^{-1})e_0,e_0\right)$. Since there are in total $2^{\omega(N)}$ positive divisors of $N$ and they correspond bijectively to non-equivalent cusps, we are done with the proof in the case of $N_1\equiv 1\imod 4$.

Now assume $N_1\equiv 3\imod 4$ and $N=4N_1$. We know that non-equivalent cusps correspond bijectively to positive divisors of $N$. Let $m_1$ be any positive divisor of $N_1$. For the cusp $s\sim \frac{1}{N/2m_1}$, let $c$ be the left lower entry of $\gamma_{2m_1}^{-1}$, hence $2||c$. Since the 2-adic Jordan component of $D$ is odd, we have $0\not\in D^{c*}$ and $\left(\rho_D(\gamma_{2m_1}^{-1})e_0,e_0\right)=0$. The other cusps, in total $2^{\omega(N)}$ of them, all give $f$ from $(F_s,e_0)|W(N)$ by the same computations.

Similarly, if $N_1\equiv 2\imod 4$ and $N=4N_1$, the cusps $s\sim \frac{1}{N/2m_1}$ and $s\sim\frac{1}{N/4m_1}$ for any $m_1\left| \frac{N_1}{2}\right.$ contribute nothing, since in these cases $0\not\in D^{c*}$ with $c$ the left lower entry of $\gamma_{2m_1}$ or $\gamma_{4m_1}$. The computation for other cusps, in total $2^{\omega(N)}$ of them, is similar, and each such cusp gives us one copy of $f$. This completes the proof.
\end{proof}

Define for each $m\imod N$, $s(m)=2^{\omega((m,N))}$.

\begin{Thm}\label{Correspondence}
The maps $\phi$ and $\psi$ are isomorphisms, inverse to each other, between $\mathcal A^\text{inv}(k,\rho_D)$ and $A^\epsilon(N,k,\chi_D)$. Explicitly, if $f=\sum_na(n)q^n\in A^\epsilon(N,k,\chi_D)$ and $\psi(f)=F=\sum_\gamma F_\gamma e_\gamma$, then
\[F_\gamma(\tau)=s(Nq(\gamma))\sum_{n\equiv Nq(\gamma)\imod N\mathbb Z}a(n)e\left(n\tau/N\right)=\sum_{n\equiv Nq(\gamma)\imod N\mathbb Z}s(n)a(n)e\left(n\tau/N\right).\]
Moreover, the same statement holds for $\mathcal A^\text{inv}(k,\rho_D^*)$ and $A^{\epsilon^*}(N,k,\chi_D)$.
\end{Thm}
\begin{proof}
That $\phi$ and $\psi$ are inverse isomorphisms follows from Proposition \ref{Injection} and Proposition \ref{Surjection}. For the explicit correspondence, let $f=\sum_na(n)q^n\in A^\epsilon(N,k,\chi_D)$, $\psi(f)=F=\sum_\gamma F_\gamma e_\gamma$, and $F_\gamma(\tau)=\sum_{n\in q(\gamma)+\mathbb Z}a(\gamma,n)q^n$ for $\gamma\in D$. We have
\[F_0|W(N)=F_0|S|V(N)=\frac{1}{\sqrt{N}}\sum_{\gamma\in D}F_\gamma|V(N)=N^{\frac{k-1}{2}}\sum_{n\in\mathbb Z}\left(\sum_{\gamma\in D: \frac{n}{N}= q(\gamma)}a\left(\gamma, nN^{-1}\right)\right)q^n.\]
Since $\phi\circ\psi=id$, we have for any $n\in\mathbb Z$,
\[\sum_{\gamma\in D: \frac{n}{N}= q(\gamma)}a\left(\gamma, nN^{-1}\right)=2^{\omega(N)}a(n).\]

Let $\gamma\in D$ with $q(\gamma)=\frac{n}{N}$. On one hand, by Proposition \ref{Invariance}, if $q(\gamma')=q(\gamma)$ then $F_\gamma=F_{\gamma'}$ and hence $a(\gamma,nN^{-1})=a(\gamma',nN^{-1})$. On the other hand,
$\gamma'$ has the same norm as $\gamma$, if and only if $\gamma$ coincides with $\gamma'$ at each $D_p$ up to permutation of $\sigma_p$, so there are $2^{\omega(N/(N_n,N))}$ elements that have the same norm as $\gamma$. Indeed, $\gamma'$ has two possibilities in $D_p$ if and only if $p\nmid n$. Then above identity implies that
\[2^{\omega(N/(N_n,N))}a(\gamma,nN^{-1})=2^{\omega(N)}a(n),\]
hence $a(\gamma,nN^{-1})=2^{\omega((n,N))}a(n)$, and the theorem follows.

The corresponding results for the spaces $\mathcal A^\text{inv}(k,\rho_D^*)$ and $A^{\epsilon^*}(N,k,\chi_D)$ follow from analogous treatment.
\end{proof}

\section{Obstructions and Rationality of Fourier Coefficients}
\noindent
In this section, we translate Borcherds's theorem of obstructions to scalar-valued modular forms using the one-to-one correspondence in the previous section.

Assume for a while that $(D,q)$ is a general discriminant form and $\rho_D$ the corresponding Weil representation of $\text{SL}_2(\mathbb Z)$ on $\mathbb C[D]$.

\begin{Lem} \label{Action}
If $F\in\mathcal A(k,\rho_D)$ (or $\mathcal M(k,\rho_D)$, respectively) and $\sigma\in\text{Aut}(D)$, then $\sigma(F)\in\mathcal A(k,\rho_D)$ (or $\mathcal M(k,\rho_D)$, respectively).
\end{Lem}
\begin{proof}
This can be seen easily from the fact that the action of $\text{SL}_2(\mathbb Z)$ and that of $\text{Aut}(D)$ on $\mathbb C[D]$ commute.
\end{proof}

We recall Borcherds's theorem on obstruction of vector-valued modular forms. We denote by $\mathcal P_D$ the space of vector-valued Fourier polynomials $P(q)=\sum_{\gamma\in D}P_\gamma(q) e_\gamma$ where each $P_\gamma(q)=\sum_{n\leq 0}c(n)q^n$ contains only finitely many terms. Let $\mathcal P^{\text{inv}}_D$ be the subspace of functions that are invariant under $\text{Aut}(D)$. In $\mathcal P_D$, we denote by $\mathcal P(k,\rho_D)$ the subspace of elements $P$ such that there exists $F\in\mathcal A(k,\rho_D)$ with $F-P$ holomorphic and vanishing at $q=0$. Similarly we have $\mathcal P^{\text{inv}}(k,\rho_D)$.

Let $P\in\mathcal P_D$ and $G\in \mathcal M(2-k,\rho_D^*)$ where $\rho_D^*$ is the dual of $\rho_D$. Assuming $P=\sum P_\gamma e_\gamma$ and $G=\sum G_\gamma e_\gamma$, we have the pairing
\[\langle P,G\rangle= \text{ the constant term of the Fourier expansion in } q \text{ of } \sum_{\gamma}P_\gamma G_\gamma.\]

\begin{Thm}[Borcherds]\label{Obstruction-Vector}
Let $P\in\mathcal P_D$. We have $P\in\mathcal P(k,\rho_D)$ if and only if  $\langle P, G\rangle=0$ for each $G\in\mathcal M(2-k,\rho_D^*)$.
\end{Thm}
\begin{proof}
This is Borcherds' Theorem 3.1 in  \cite{borcherds1999gross}.\end{proof}

\begin{Cor}\label{Obstruction-Invariant}
Let $P\in \mathcal P_D^\text{inv}$. Then $P\in\mathcal P^\text{inv}(k,\rho_D)$ if and only if $\langle P, G\rangle=0$ for each $G\in\mathcal M^\text{inv}(2-k,\rho_D^*)$.
\end{Cor}
\begin{proof}
The forward direction follows directly from Theorem \ref{Obstruction-Vector}.
Now assume $P\in\mathcal P_D^\text{inv}$ and $\langle P,G\rangle=0$ for each $G\in\mathcal M^\text{inv}(2-k,\rho_D^*)$.

It is easy to see that the pairing $\langle\cdot,\cdot\rangle$ is invariant under $\text{Aut}(D)$. For any $G\in \mathcal M(2-k,\rho_D^*)$ and any $\sigma\in\text{Aut}(D)$, we have $\langle \sigma P, \sigma G\rangle=\langle P, \sigma G\rangle$. So it follows that
\[\langle P, G\rangle=\langle P, G'\rangle,\quad G'=\frac{1}{|\text{Aut}(D)|}\sum_{\sigma\in\text{Aut}(D)}\sigma G.\]
By Lemma \ref{Action}, $G'\in \mathcal M(2-k,\rho_D^*)$, and clearly $G'$ is invariant under $\text{Aut}(D)$. Therefore $\langle P, G\rangle=\langle P, G'\rangle=0$ by the assumption. This implies that $P\in\mathcal P(k,\rho_D)$ by Theorem \ref{Obstruction-Vector}. So there exists $F\in\mathcal A(k,\rho_D)$ such that $F-P$ vanishes at $q=0$. Now consider $F'=\frac{1}{|\text{Aut}(D)|}\sum_{\sigma\in\text{Aut}(D)}\sigma F.$
It is clear that $F'-P$ vanishes at $q=0$ and $F'\in\mathcal A^\text{inv}(k,\rho_D)$, and hence $P\in\mathcal P^\text{inv}(k,\rho_D)$.
\end{proof}

We return to our previous setting now and assume that $k\leq 0$ and hence $2-k\geq 2$.
Let us denote by $E(N,2-k,\chi_D)$ the space of Eisenstein series of level $N$, weight $2-k$ and character $\chi_D$. It is well-known that $\text{dim}(E(N,2-k,\chi_D))=2^{\omega(N)}$, with a basis concretely given by $\{E_m: m\mid N, m=N_m\}$ where (see Theorem 4.5.2 and Theorem 4.6.2 in \cite{diamond2005first})
\[
E_m= \delta_{1,m}L\left(k-1,\chi_D\right)+2\sum_{n=1}^\infty\left(\sum_{d\mid n}\chi_m(n/d)\chi_m'(d)d^{1-k}\right)q^n.
\]
Here $\delta_{1,m}=1$ if $m=1$ and $0$ otherwise. Let $E^{\epsilon^*}=L(k-1,\chi_D)^{-1}\sum_{m}E_m$. It can be seen easily that $E^{\epsilon^*}\in M^{\epsilon^*}(N,2-k,\chi_D)$.

\begin{Lem}\label{Eisenstein}
For each $\delta$, we have $\text{dim}(E^{\delta}(N,2-k,\chi_D))=1$. In particular, $E^{\epsilon^*}(N,2-k,\chi_D)=\text{span}_\mathbb{C}\{E^{\epsilon^*}\}$.
\end{Lem}
\begin{proof}
We will prove that $\text{dim}(E^{\delta}(N,2-k,\chi_D))\geq 1$, from which the lemma follows.

We consider $f=E_1$ and maintain the notations in the proof of Proposition \ref{Decomposition}. We have $g\in M^\delta(N,2-k,\chi_D)$ and the coefficients of $g$ when $(n,N)=1$ are given by
\[b(n)=2^{-\omega(N)}a(n)\prod_{p\mid N}(1+\delta_p\chi_p(n)).\] So it suffices to show that $g\neq 0$. To this end, let us fix a set of integers $\{c_p: p\mid N\}$ such that $\chi_p(c_p)=\delta_p$. We then may choose a prime $q$ such that $q\equiv c_p\imod N_p$ for each $p\mid N$. Therefore, $\chi_p(q)=\delta_p$. On the other hand, $a(q)=2(\chi_D(q)q^{1-k}+1)\neq 0$. Therefore, $b(q)\neq 0$ and we are done.
\end{proof}

Now assume $E^{\epsilon^*}=\sum_{n\geq 0} B(n)q^n$. Here $B(0)=1$. We have the obstruction theorem for scalar-valued modular forms. (See Theorem 6 in \cite{bruinier2003borcherds} in the case of prime level.)

\begin{Thm}\label{Obstruction-Scalar}
Let $P=\sum_{n<0}a(n)q^n$ be a polynomial in $q^{-1}$ such that its coefficients satisfy the $\epsilon$-condition. Then there exists $f\in A^\epsilon(N,k,\chi_D)$ with $f=\sum_{n\in\mathbb Z}a(n)q^n$, if and only if
\[\sum_{n<0}s(n)a(n)b(-n)=0,\]
for each $g=\sum_{n\geq 0}b(n)q^n\in S^{\epsilon^*}(N,2-k,\chi_D)$. If $N_1\equiv 1,3\imod 4$ and $f$ exists, then $f$ is unique and its constant term is given by
\[a(0)=-\frac{1}{s(0)}\sum_{n< 0}s(n)a(n)B(-n).\]
\end{Thm}
\begin{proof}
By Theorem \ref{Correspondence}, we have
\[\psi(f)_\gamma=s(Nq(\gamma))\sum_{n\equiv Nq(\gamma)\imod N\mathbb Z} a(n)e\left(\frac{n\tau}{N}\right),\]
\[\psi(g)_\gamma=s(Nq(\gamma))\sum_{n\equiv -Nq(\gamma)\imod N\mathbb Z} b(n)e\left(\frac{n\tau}{N}\right).\]
Now $\rho_D^*$ is equal to $\rho_{D[-1]}$ where $D[-1]$ is the module $D$ with discriminant form $-q(\cdot)$. Then by Corollary \ref{Obstruction-Invariant}, we have that $f$ exists if and only if $\langle P, \psi(g)\rangle=0$ for each $g$. Concretely, it means the constant term of
\[\sum_\gamma s(Nq(\gamma))^2\left(\sum_{n\equiv Nq(\gamma)\imod N\mathbb Z} a(n)e\left(\frac{n\tau}{N}\right)\right)\left(\sum_{n\equiv -Nq(\gamma)\imod N\mathbb Z} a(n)e\left(\frac{n\tau}{N}\right)\right)\]
vanishes. The constant term is given by
\[\sum_\gamma s(Nq(\gamma))^2\sum_{n\in\mathbb Z}a(N(-n+q(\gamma)))b(N(n-q(\gamma))),\] which in turn simplifies to
\[\sum_{n\in\mathbb Z}2^{\omega(N)}s(n)a(-n)b(n),\]
from which the first part follows.

By Lemma \ref{Eisenstein} and the same computation, we have the constant term expression assuming the existence.
The uniqueness follows from Corollary \ref{Holomorphy} and the assumption $k\leq 0$.
\end{proof}

Rationality of the Fourier coefficients is important in Borcherds's theory of automorphic products. We end this section with the rationality results, following the lines in \cite{bruinier2003borcherds}. For $f=\sum_n a(n)q^n$ and $\sigma\in\text{Gal}(\mathbb C/\mathbb Q)$, define $f^\sigma=\sum_n a(n)^\sigma q^n$. Let $k$ be an even integer.

\begin{Lem}\label{Galois}
If $f\in A(N,k,\chi_D)$, so is $f^\sigma$.
\end{Lem}
\begin{proof}
It is well-known that $M(N,k,\chi_D)$ has a basis of forms with rational integral Fourier coefficients for any even positive weight $k$. (See Corollary 12.3.8 and Proposition 12.3.11 in \cite{diamond1995modular}.)

Take a large positive integer $k'$ and we know $f\Delta^{k'}\in M(N,k+12k',\chi_D)$. The above observation shows that $(f\Delta^{k'})^\sigma\in M(N,k+12k',\chi_D)$. But $\Delta$ has rational integral Fourier coefficients, hence $f^\sigma \Delta^{k'}=(f\Delta^{k'})^\sigma \in M(N,k+12k',\chi_D)$ and $f^\sigma \in A(N,k,\chi_D)$.
\end{proof}

\begin{Prop}
Let $k\leq 0$ be an even integer and assume $N_1\equiv 1,3\imod 4$. Let $f=\sum_na(n)q^n\in A^\delta(N,k,\chi_D)$ and suppose that $a(n)\in\mathbb Q$ for $n<0$. Then all coefficients $a(n)$ are rational with bounded denominator.
\end{Prop}
\begin{proof}
Let $\sigma\in\text{Gal}(\mathbb C/\mathbb Q)$. We note that $f^\sigma\in A^\delta(N,k,\chi_D)$. Indeed, from Lemma \ref{Galois}, we see that $f^\sigma\in A(N,k,\chi_D)$; moreover, the Galois action preserves the $\delta$-condition.

Now consider $h=f-f^\sigma\in A^\delta(N,k,\chi_D)$. It is obvious that $h$ is holomorphic at $\infty$, hence $h\in M(N,k,\chi_D)$ by Corollary \ref{Holomorphy}. But $k\leq 0$, so $M(N,k,\chi_D)=0$. It follows that $f$ has rational coefficients. Since $f\Delta^{k'}\in M(N,k+12k',\chi_D)$ for large $k'$, it has coefficients with bounded denominator, hence so does $f$.
\end{proof}

\section{An Example}
\noindent
As an application, we consider the case $N_1=3$.

We know that $S(12,2,\chi_D)=0$ (see, for example, Chapter 6 of Stein's book \cite{stein2007modular}). Therefore, by Theorem \ref{Obstruction-Scalar}, for any polynomial $P$ in $q^{-1}$ without constant term that satisfies the $\epsilon$-condition, there exists $f\in A^\epsilon(N,0,\chi_D)$ such that $f-P$ is holomorphic at $q=0$. In this case, we say $P$ is the \emph{principal part} of $f$.

Explicitly, we have in this case $\epsilon^*_2=\epsilon^*_3=1$ and $\epsilon_2=\epsilon_3=-1$. Let $m$ be a positive integer such that $\chi_2(-m)\neq 1$ and $\chi_3(-m)\neq 1$; it means $m\not\equiv -1\imod 4$ and $m\not\equiv -1\imod 3$. For each such $m$, $2^{-\omega((m,12))}q^{-m}$ satisfies the $\epsilon$-condition, hence there exists uniquely
\[f_m=2^{-\omega((m,12))}q^{-m}+\sum_{n\geq 0}a(n)q^n\in A^\epsilon(12,0,\chi_D).\]
In particular, $f_1$ exists. We construct it as follows.

We end this section with the construction of $f_1$. Let $E_2=1-24\sum_{n=1}^\infty\sigma_1(n)q^n$ with $\sigma_1(n)=\sum_{0<d\mid n}d$. Now let \[\frak E_2(\tau)=\frac{1}{24}\left(E_2(\tau)-9E_2(3\tau)-4E_2(4\tau)+36E_2(12\tau)\right).\] It is clear that $\frak E_2\in M(12,2,1)$, even though $E_2$ itself is not a modular form. Now consider the $\eta$-quotient
\[H_2(\tau)=\eta(\tau)^2\eta(3\tau)^{-2}\eta(4\tau)\eta(6\tau)^2\eta(12\tau).\]
By Table I in \cite{martin1996multiplicative}, $H_2\in M(12,2,\chi_D)$ and $H_2$ is vanishing at $\infty$.

\begin{Prop} We have
$f_1=\frak E_2/H_2$.
\end{Prop}
\begin{proof}
We first investigate the behavior of $\frak E_2/H_2$ at cusps other than $\infty$. By the formula given in \cite{martin1996multiplicative}, we compile the data of $H_2$ in the following table:
\[
\begin{array}{c|cccccc}
\text{cusp } s & \infty & 0 & 1/3 & 1/4 & 1/2 & 1/6\\
\hline
\text{first exponent in } q_s& 1&1 &0& 1&1/2 & 1/2\\
\end{array}
\] where $q_s$ is the uniformizer at the cusp $s$. On the other hand, we may rewrite $\frak E_2$ as
\[24\frak E_2(\tau)=\mathbb E_2(\tau)-9\mathbb E_2(3\tau)-4\mathbb E_2(4\tau)+36\mathbb E_2(12\tau),\] where $\mathbb E_2(\tau)=E_2(\tau)-\frac{3}{\pi\text{Im}(\tau)}$ is a (real analytic) modular form of weight $(2,0)$. From this, we can easily tell that $\frak E_2$ vanishes at $0$ and $1/4$ and hence $\frak E_2/H_2$ is holomorphic at all cusps except $\infty$, while the Fourier expansion of $\frak E_2/H_2$ at $\infty$ begins with $q^{-1}$.

Since $f_1$ contains a single negative power term $q^{-1}$, by Corollary \ref{Holomorphy-Special}, we know that $f_1$ is holomorphic at all cusps other than $\infty$. Since the Fourier expansion of $\frak E_2/H_2$ at $\infty$ begins with $q^{-1}$ as well and it is also holomorphic at all other cusps, $f_1-\frak E_2/H_2\in M(12,0,\chi_D)$, hence $f_1=\frak E_2/H_2$.
\end{proof}

The first few terms of the Fourier expansion of $f_1$ at $\infty$ are
\[f_1=\frac{1}{q}+1+2q^2+q^3-2q^6-2q^8+4q^{12}+4q^{14}-q^{15}-6q^{18}+O(q^{19}).\]

\vskip 0.5 cm

\addcontentsline{toc}{chapter}{Bibliography}
\bibliographystyle{amsplain}
\bibliography{paper}

\end{document}